\documentclass[a4paper,11pt]{article}

\usepackage[latin1]{inputenc}
\usepackage[T1]{fontenc}
\usepackage[normalem]{ulem}
\usepackage[english]{babel}
\usepackage{verbatim}
\usepackage{graphicx}
\usepackage{enumitem}
\usepackage{amsmath,amssymb,amsfonts,amsthm}
\usepackage{rotating}
\usepackage{float}
\usepackage{array}

\newtheorem{theorem}{Theorem}[section]
\newtheorem*{maintheorem}{Main Theorem}
\newtheorem{lemma}[theorem]{Lemma}
\newtheorem{prop}[theorem]{Proposition}
\newtheorem{cor}[theorem]{Corollary}

\theoremstyle{definition}
\newtheorem{defn}[theorem]{Definition}
\newtheorem{rem}[theorem]{Remark}
\newtheorem{ex}[theorem]{Example}

\input xy
\xyoption{all}
\newcommand{\sbt}{\,\begin{picture}(-1,1)(0.5,-1)\circle*{1.8}\end{picture}\hspace{.05cm}}

\DeclareMathOperator{\id}{id}

\DeclareMathOperator{\tr}{Tr}

\DeclareMathOperator{\THH}{THH}

\DeclareMathOperator{\hofib}{hof}
\DeclareMathOperator{\hoc}{hoc}

\DeclareMathOperator{\KR}{KR}

\DeclareMathOperator{\K}{K}
\DeclareMathOperator{\TC}{TC}
\DeclareMathOperator{\Sp}{Sp}

\newdir^{ (}{{}*!/-5pt/@^{(}}

\begin{document}
\begin{center}\LARGE{A Dundas-McCarthy theorem for bimodules over exact categories}
\end{center}

\begin{center}\large{Emanuele Dotto}
\end{center}
\vspace{.3cm}

\begin{quote}
\textsc{Abstract} From a bimodule $M$ over an exact category $C$, we define an exact category $C\ltimes M$ with a projection down to $C$. This construction classifies certain split square zero extensions of exact categories. We show that the trace map induces an equivalence between the relative $K$-theory of $C\ltimes M$ and its relative topological cyclic homology. When applied to the bimodule $\hom(-,-\otimes _AM)$ on the category of finitely generated projective modules over a ring $A$ one recovers the classical Dundas-McCarthy theorem for split square zero extensions of rings.
\end{quote}

\tableofcontents

\section*{Introduction}

A bimodule over a category $C$ enriched in Abelian groups is an enriched functor $C^{op}\otimes C\rightarrow Ab$ to the category of Abelian groups. We define the semi-direct product of $C$ and $M$ as the category $C\ltimes M$ with the same objects as $C$, morphism groups
\[(C\ltimes M)(a,b)=C(a,b)\oplus M(a,b)\]
and composition defined by using the bimodule structure. A category enriched in Abelian groups with one object is precisely a ring $A$, and this construction generalizes the classical semi-direct product of rings and bimodules $A\ltimes M$. Moreover the groupoid of isomorphisms $i(C\ltimes M)$ is isomorphic to the Grothendieck construction of the composite functor
\[iC\stackrel{\Delta}{\longrightarrow}(iC)^{op}\times iC\stackrel{M}{\longrightarrow}Ab\]
where $\Delta$ sends a morphism $f$ in $C$ to $(f^{-1},f)$.

If $C$ is an exact category and $M$ is exact in both variables, we show in \ref{CltimesMexact} that $C\ltimes M$ has an induced exact structure. Projecting off the module summand on morphisms gives an exact functor $C\ltimes M\rightarrow C$. The aim of this paper is to compare the relative $K$-theory of $C\ltimes M\rightarrow C$ and its relative topological cyclic homology, namely:
\begin{maintheorem}
The square of spectra
\[\xymatrix{K(C\ltimes M)\ar[d]\ar[r]&\TC(C\ltimes M)\ar[d]\\
K(C)\ar[r]&\TC(C)
}\]
is homotopy cartesian.
\end{maintheorem} 
The proof of this result relies heavily on the classical proof of the Dundas-McCarthy theorem of \cite{McCarthy} for split square zero extensions of rings. In fact, the proof of \cite{McCarthy} translates to this situation with very minor changes, once we decompose the $K$-theory of $C\ltimes M$ as a disjoint union
\[K(C\ltimes M)\simeq \coprod_{c\in S_{\sbt}C}S_{\sbt}M(c,c)\]
where $S_{\sbt}M$ is an extension of the bimodule $M$ to the $S_{\sbt}$-construction of $C$. This is proved in \ref{splitKRmodel}, using techniques similar to \cite[1.2.5]{DGM}. The main point is that the zero section induces a levelwise essentially surjective functor $S_{\sbt}C\rightarrow S_{\sbt}(C\ltimes M)$, which might be surprising when $C$ is not split-exact.  

In \S\ref{cltimesmsec} we show that the bimodules over $C$ classify all the ``split square zero extensions of exact categories'' over $C$ via the semi-direct product construction described above. For the category $\mathcal{P}_A$ of finitely generated projective modules over a ring $A$, and an $A$-bimodule M, this gives an equivalence of categories
\[\mathcal{P}_{A\ltimes M}\simeq \mathcal{P}_A\ltimes \hom_A(-,-\otimes _AM)\]
over $\mathcal{P}_A$. Specified to this situation the  main theorem above recovers the Dundas-McCarty theorem of \cite{McCarthy} for the projection map of rings $A\ltimes M\rightarrow A$.

\section{Split square zero extensions of exact categories and bimodules}\label{cltimesmsec}

Let $Ab$ be the symmetric monoidal category of Abelian groups and tensor product.

\begin{defn}
A split extension of $Ab$-enriched categories is a triple $(p,s,U)$ of enriched functors $p\colon C'\longrightarrow C$ and $s\colon C\longrightarrow C'$ and a natural isomorphism $U\colon p\circ s\Rightarrow \id_{C}$, such that the section $s\colon C\longrightarrow C'$ is essentially surjective.

We call $(p,s,U)$ square zero if for every pair of composable morphisms $f\colon s(a)\rightarrow s(b)$ and $g\colon s(b)\rightarrow s(c)$ in $C'$ with $p(f)=0$ and $p(g)=0$ in $C$, the composite $g\circ f$ is zero in $C'$.

A split extension of exact categories is a split extension $(p,s,U)$ with $C$ and $C'$ exact categories and $p$ and $s$ exact functors.
\end{defn}
Notice that the isomorphism $U\colon p\circ s\Rightarrow \id_{C}$ insures that $s$ is essentially injective. Since it is assumed to be essentially surjective, it follows that $C'$ and $C$ have essentially the same objects.

\begin{ex}\label{sqzerorings}
Split square zero extensions of $Ab$-enriched categories with one object are precisely the split square-zero extensions of rings $f\colon R\longrightarrow A$. These are all of the form $A\ltimes M\longrightarrow A$ for some $A$-bimodule $M$, where $A\ltimes M$ is the Abelian group $A\oplus M$ with multiplication
\[(a,m)\cdot(b,n)=(ab,an+mb)\]
What follows is the generalization for exact categories.
\end{ex}

\begin{defn} A bimodule over an exact category $C$ is an additive functor $M\colon C^{op}\otimes C\longrightarrow Ab$ which is exact in both variables. In particular $M(0,c)=M(c,0)=0$.
For a morphism $f\colon a\rightarrow b$ and an object $c$ in $C$ we denote
\[f_\ast=M(\id_c\otimes f)\colon M(c,a)\rightarrow M(c,b) \ \ \ ,\ \ \ f^\ast=M(f\otimes \id_c)\colon M(b,c)\rightarrow M(a,c)\]
the induced maps.
\end{defn}

Given a bimodule $M$ over $C$, define a category enriched in Abelian groups $C\ltimes M$ with the same objects of $C$, and morphism groups $(C\ltimes M)(a,b)=C(a,b)\oplus M(a,b)$. Composition is defined by
\[\big(f\colon b\rightarrow c,m\in M(b,c)\big)\circ \big(g\colon a\rightarrow b,n\in M(a,b)\big)=\big(f\circ g,f_\ast n+g^\ast m\big)\]
Projection onto the morphisms of $C$ defines a functor $p\colon C\ltimes M\longrightarrow C$. This is split by the functor $s\colon C\longrightarrow C\ltimes M$ defined by inclusion at zero, with strict composition $p\circ s=\id$.

\begin{rem}\label{ltimesGrothendieck}
If $C$ is a groupoid, there is a functor $\Delta\colon C\rightarrow C^{op}\times C$ that sends $f\colon c\rightarrow d$ to $(f^{-1},f)\colon (c,c)\rightarrow (d,d)$. The category $C\ltimes M$ is then the Grothendieck construction of the composite $M\circ\Delta\colon C\rightarrow C^{op}\times C\rightarrow Ab$, in symbols
\[C\ltimes M=C\wr (M\circ \Delta)\]
In particular by denoting $iC$ the groupoid of isomorphisms of $C$ we have
\[i(C\ltimes M)=(iC)\ltimes M=(iC)\wr (M\circ \Delta)\]
where we denoted the restriction $(iC)^{op}\times iC\rightarrow Ab$ also by $M$.
\end{rem}

\begin{ex}
Considering a ring $A$ as an $Ab$-enriched category with one object, an $Ab$-enriched functor $M\colon A^{op}\otimes A\longrightarrow Ab$ is precisely a bimodule over $A$ in the classical sense.
The construction $A\ltimes M$ above is the classical semidirect product of \ref{sqzerorings}.
\end{ex}

\begin{ex}
Let $\mathcal{P}_A$ be the exact category of finitely generated projective modules over a ring $A$, and let $M$ be an $A$-bimodule. The functor
\[\hom_A(-,-\otimes_AM)\colon (\mathcal{P}_A)^{op}\otimes \mathcal{P}_A\longrightarrow Ab\]
is a bimodule on $\mathcal{P}_A$. It is exact in both variables since $\mathcal{P}_A$ is a split-exact category.
\end{ex}

\begin{defn}\label{defexactCM} 
A sequence $a\stackrel{(i,m)}{\longrightarrow}b\stackrel{(q,n)}{\longrightarrow}c$ of $C\ltimes M$ is exact if the underlying sequence $a\stackrel{i}{\longrightarrow}b\stackrel{q}{\longrightarrow}c$ is exact in $C$, and \[q_\ast m+i^\ast n=0\ \ \in M(a,c)\]
\end{defn}

\begin{prop}\label{CltimesMexact}
The exact sequences of \ref{defexactCM} define an exact structure on $C\ltimes M$, and the triple $(p,s,\id)$ is a split square zero extension of exact categories.
\end{prop}

\begin{proof} Let us prove the non obvious axioms for the exact structure of $C\ltimes M$. 
First we need to show that if $a\stackrel{(i,m)}{\longrightarrow}b\stackrel{(q,n)}{\longrightarrow}c$ is exact then $(i,m)$ is a kernel for $(q,n)$, and that $(q,n)$ is a cokernel for $(i,m)$. The composite $(q,n)\circ(i,m)=(q\circ i,q_\ast m+i^\ast n)$ is zero by definition. We need to show that given any $(f,l)\colon d\rightarrow b$ with 
\[(q,n)\circ (f,l)=(q\circ f,q_\ast l+f^{\ast}n)=0\]
 there is a unique map $(u,x)\colon d\rightarrow a$ with $(i,m)\circ (u,x)=(f,l)$. Let $u\colon d\rightarrow b$ be the universal map in $C$ for the kernel $i\colon a\rightarrow b$ of $q\colon b\rightarrow c$. The element $x\in M(d,a)$ must be the unique solution for the equation
 \[l=u^\ast m+i_\ast x\]
Since $i_\ast\colon M(d,a)\rightarrow M(d,b)$ is injective by exactness of $M$, if a solution exists it must be unique. For the existence, $l-u^\ast m$ must belong to the image of $i_\ast$ which is equal to the kernel of $q_\ast\colon M(d,b)\rightarrow M(d,c)$. We compute
\[q_\ast(l-u^\ast m)=q_\ast l-q_\ast u^\ast m=-f^{\ast}n-u^\ast q_\ast m =-f^{\ast}n+u^\ast i^\ast n=-f^{\ast}n+f^\ast n=0\]
where $-q_\ast m =i^\ast n$ by exactness. A dual proof shows that $(q,n)$ is a cokernel for $(i,m)$.

Now we show that pullbacks of admissible epimorphisms exist in $C\ltimes M$, and are admissible epimorphisms. Given a diagram $c\stackrel{(g,n)}{\rightarrow} b\stackrel{(f,m)}{\twoheadleftarrow} a$ in $C\ltimes M$, let 
\[\xymatrix{p\ar[r]^{\overline{g}}\ar@{->>}[d]_{\overline{f}}&a\ar@{->>}[d]^{f}\\
c\ar[r]_g&b
}\]
be the pullback of the underlying diagram in $C$. First we find $\overline{m}\in M(p,c)$ and $\overline{n}\in M(p,a)$ such that
\[\xymatrix{p
\ar[r]^{(\overline{g},\overline{n})}\ar@{->>}[d]_{(\overline{f},\overline{m})}&a\ar@{->>}[d]^{(f,m)}\\
c\ar[r]_{(g,n)}&b
}\]
commutes. This is the same as the condition $f_\ast\overline{n}+\overline{g}^\ast m=g_\ast \overline{m}+\overline{f}^\ast n$.
Choose $\overline{m}=0$ and $\overline{n}$ in the preimage of $\overline{f}^\ast n-\overline{g}^\ast m$ by the surjective map $f_\ast\colon M(p,a)\twoheadrightarrow M(p,b)$.
For $p$ equipped with these maps, we show the universal property
\[\xymatrix{d\ar@{-->}[r]^{(u,z)}\ar@/^2pc/[rr]^{(k,x)}\ar[dr]_{(h,y)}&p
\ar[r]^{(\overline{g},\overline{n})}\ar@{->>}[d]^{(\overline{f},0)}&a\ar@{->>}[d]^{(f,m)}\\
&c\ar[r]_{(g,n)}&b
}\]
Let $u\colon d\rightarrow p$ be the universal map for the corresponding diagram in $c$. The element $z\in M(d,p)$ must by the unique solution to the equations
\[ \overline{g}_\ast z=x-u^\ast \overline{n}\ \ \ \ \ \mbox{and}\ \ \ \ \ \overline{f}_\ast z=y \]
Notice that $x-u^\ast\overline{n}$ and $y$ lie over the same element of $M(d,b)$, as
\[f_{\ast}(x-u^\ast\overline{n})=f_\ast x-u^\ast f_\ast\overline{n}=f_\ast x-u^\ast (-\overline{g}^\ast m+\overline{f}^\ast n)=f_\ast x+k^\ast m-h^\ast n=g_\ast y\]
where the last equality is the commutativity of the outer diagram.
Therefore a solution $z$ to the equations above exists and is unique precisely when the map $(\overline{f}_\ast,\overline{g}_\ast)\colon M(d,p)\rightarrow M(d,c)\times_{M(d,b)}\times M(d,a)$ is an isomorphism, that is when
\[\xymatrix{M(d,p)\ar[r]^{\overline{g}_\ast}
\ar@{->>}[d]^{\overline{f}_\ast}
& M(d,a)\ar@{->>}[d]^{f_\ast}\\
M(d,c)\ar[r]_{g_\ast}&M(d,b)
}\]
is a pullback square. This follows from exactness of $M$ in the second variable, and from the fact that kernels for $f$ and $\overline{f}$ in $C$ are isomorphic. To show that $(\overline{f},0)$ is an admissible epimorphism choose an exact sequence $k\stackrel{i}{\rightarrowtail}p\stackrel{f}{\twoheadrightarrow}c$ in $C$. Then $k\stackrel{(i,0)}{\rightarrowtail}p\stackrel{(f,0)}{\twoheadrightarrow}c$ is clearly exact in $C\ltimes M$.
A dual argument shows the analogous statement for pushouts and admissible monomorphisms.
The other axioms follow easily from the axioms for $C$.

The functors $p$ and $s$ are clearly exact. To see that $(p,s,U)$ is square zero, we need to show that the composition of morphisms of the form $(0,m)$ is zero. This is clear.
\end{proof}

Given a split extension $(p,s,U)$, let $\ker p\colon C^{op}\otimes C\longrightarrow Ab$ be the bimodule defined by
\[\ker p(a,b)=\ker\big(C'(s(a),s(b))\stackrel{p}{\longrightarrow} C(ps(a),ps(b))\big)\]
\begin{rem}
This bimodule is always left exact in both variables. It is also right exact provided the $\hom$-functors of $C$ and $C'$ are also right exact.
\end{rem}
If the extension $(p,s,U)$ is square zero, there is a canonical functor $F\colon C\ltimes \ker p\longrightarrow C'$ that sends an object $c$ to $s(c)$, and a morphism $(f,m)$ to $s(f)+m$. The square zero condition is used for showing that $F$ preserves composition. Moreover the diagram
\[\xymatrix{C\ltimes \ker p\ar[r]^-{F}\ar[d]&C'\ar[d]^{p}\\
C\ar[r]^{\simeq}_{p\circ s}&C}
\]
commutes, with the functor $p\circ s$ an equivalence of exact categories.

\begin{prop}\label{classsplitext}
Let $(p,s,U)$ be a split square zero extension of $Ab$-enriched categories. The functor $F\colon C\ltimes \ker p\longrightarrow C'$ is an equivalence of $Ab$-enriched categories.

I moreover $p$ and $s$ are exact and the $\hom$-functors of $C$ and $C'$ are right exact, $F\colon C\ltimes \ker p\longrightarrow C'$ is an equivalence of exact categories for the exact structure on $C\ltimes \ker p$ of \ref{CltimesMexact}.
\end{prop}

\begin{proof}
The functor $F$ is obviously essentially surjective, since $s$ is by assumption. To see that it is fully faithful, define an inverse for
\[F\colon C(a,b)\oplus \ker p(a,b)\longrightarrow C'(s(a),s(b))\]
by sending $f\colon s(a)\longrightarrow s(b)$ to the pair $(U^{\ast}p(f),f-sp(f))$ where $U^{\ast}\colon C(ps(a),ps(b))\longrightarrow C(a,b)$ is the isomorphism induced by the isomorphism $U\colon ps\Rightarrow\id$.

We show that $F$ is exact. By lemma \ref{reductionexactseq} one can choose an isomorphism between any exact sequence $E$ in $C\ltimes \ker p$ and one of the form $a\stackrel{(i,0)}{\longrightarrow}b\stackrel{(q,0)}{\longrightarrow}c$. This gives an isomorphism between $F(E)$ and
\[F\big(a\stackrel{(i,0)}{\longrightarrow}b\stackrel{(q,0)}{\longrightarrow}c\big)=\big(s(a)\stackrel{s(i)}{\longrightarrow}s(b)\stackrel{s(q)}{\longrightarrow}s(c)\big)\]
which is exact by exactness of $s$.

It remains to show that if $E=\big(a\stackrel{(i,m)}{\longrightarrow}b\stackrel{(q,n)}{\longrightarrow}c\big)$ is any sequence of $C\ltimes \ker p$ whose image $F(E)$ is exact in $C'$, then $E$ is exact in $C\ltimes \ker p$. If $F(E)$ is exact the underlying sequence
\[\big(a\stackrel{i}{\longrightarrow}b\stackrel{q}{\longrightarrow}c\big)=pF(E)\]
is also exact. Moreover the composition $F(q,n)\circ F(i,m)$ is zero, and since $F$ is faithful we must have $(q,n)\circ (i,m)=0$ in $C\ltimes \ker p$, that is $q_\ast m+i^{\ast}n=0$.
\end{proof}

\begin{lemma}\label{reductionexactseq}
Any exact sequence of $C\ltimes M$ is isomorphic (non canonically) to an exact sequence of the form $a\stackrel{(i,0)}{\longrightarrow}b\stackrel{(q,0)}{\longrightarrow}c$. 
\end{lemma}

\begin{proof}
Let $a\stackrel{(i,m)}{\longrightarrow}b\stackrel{(q,n)}{\longrightarrow}c$ be an exact sequence of $C\ltimes M$. We find an isomorphism of the form
\[\xymatrix{a\ar[d]_{(\id,0)}\ar[r]^{(i,m)}&b\ar[d]_{(\id,x)}
\ar[r]^{(q,n)}&c\ar[d]^{(\id,0)}\\
a\ar[r]_{(i,0)}&b\ar[r]_{(q,0)}&c
}\]
for some $x\in M(b,b)$. The vertical maps are isomorphisms for any $x\in M(b,b)$. For the squares to commute, $x$ needs to satisfy
\[i^\ast x=-m\ \ \ \ \ \mbox{and }\ \ \ \ \ q_\ast x=n\]
This is the case precisely when the element $(-m,n)$ is in the image of
\[(i^{\ast},q_{\ast})\colon M(b,b)\longrightarrow M(a,b)\times_{M(a,c)} M(b,c)\]
Consider the diagram of Abelian groups
\[\xymatrix{
M(c,b)\ar@{->>}[r]^{q_\ast}\ar@{>->}[d]_{q^\ast}&M(c,c)\ar@{=}[r]\ar@{>->}[d]&M(c,c)\ar@{>->}[d]^{q^{\ast}}\\
M(b,b)\ar[r]^-{(i^{\ast},q_{\ast})}\ar@{->>}[d]_{j^\ast}&M(a,b)\times_{M(a,c)} M(b,c)\ar[r]\ar@{->>}[d]&M(b,c)\ar@{->>}[d]^{j^\ast}\\
M(a,b)\ar@{=}[r]& M(a,b)\ar@{->>}[r]_{q_\ast}&M(a,c)
}
\]
The middle column is exact as pullback of an exact sequence. Therefore $(i^{\ast},q_{\ast})$ is surjective, and a lift of $(-m,n)$ exists.
\end{proof}

\begin{ex}\label{modulesoverltimes}
Let $A$ be a ring, $M$ an $A$-bimodule, and $p\colon A\ltimes M\rightarrow A$ the projection with zero section $s\colon A\rightarrow A\ltimes M$. The induced functors $p_\ast=(-)\otimes_{A\ltimes M}A$ and $s_{\ast}=(-)\otimes_{A}(A\ltimes M)$ define a split square zero extension of exact categories ($s_\ast$ is essentially surjective by \cite[1.2.5.4]{DGM}). Therefore \ref{classsplitext} provides an equivalence of categories
\[\mathcal{P}_{A\ltimes M}\simeq\mathcal{P}_A\ltimes \ker p_\ast\]
For $Q\in\mathcal{P}_A$ there is a natural isomorphism
\[s_\ast Q\cong Q\oplus (Q\otimes_AM)\]
where the right-hand side has module structure $(p',p\otimes n)\cdot(a,m)=(p'a, p\otimes na+p\otimes m)$. Under this isomorphism a module map $s_\ast Q\rightarrow s_\ast Q'$ is a matrix $\bigl(\begin{smallmatrix} \phi&0\\ f&\phi\otimes M \end{smallmatrix} \bigr)$ for module maps $\phi\colon Q\rightarrow Q'$ and $f\colon Q\rightarrow Q'\otimes_A M$ (see \cite[1.2.5.1]{DGM}). This gives a natural isomorphism of bimodules $\ker p_\ast\cong \hom_A(-,-\otimes_AM)$, and combined with \ref{classsplitext} an equivalence of exact categories
\[\mathcal{P}_{A\ltimes M}\simeq\mathcal{P}_A\ltimes \hom_A(-,-\otimes_AM)\]
\end{ex}

\section{The K-theory of $C\ltimes M$}

Let $C$ be an exact category and $M\colon C^{op}\otimes C\longrightarrow Ab$ a bimodule over $C$.
We aim at describing the $K$-theory $\K(C\ltimes M)$.
We recall from \cite[\S 1.3.3]{DGM} that for every natural number $p$, the bimodule $M$ extends to a bimodule $S_{p}M\colon S_{p}C^{op}\otimes S_{p}C\longrightarrow Ab$ on $S_{p}C$, defined as the Abelian subgroup
\[S_{p}M(X,Y)\subset \bigoplus\limits_{\theta\in Cat([1],[p])}M(X_\theta,Y_{\theta})
\]
of collections $\underline{m}$ satisfying the condition
\[X(a)^{\ast}m_\theta=Y(a)_\ast m_\rho\]
as elements of $M(X_\rho,Y_\theta)$, for all morphism $a\colon \rho\rightarrow\theta$ in $Cat([1],[p])$. In the terminology of \S\ref{cltimesmsec} this condition expresses the commutativity of the diagram
\[\xymatrix{X_{\rho}\ar[d]_{(X(a),0)}\ar[r]^{(0,m_\rho)}&Y_\rho
\ar[d]^{(Y(a),0)}\\
X_\theta\ar[r]_{(0,m_{\theta})}&Y_\theta
}\]
in $C\ltimes M$.
Iterating this construction defines for all $\underline{p}=(p_1,\dots,p_n)$  a bimodule on $S^{(n)}_{\underline{p}}C$
\[S^{(n)}_{\underline{p}}M=S_{p_1}\dots S_{p_n}M\]
A map $\sigma\colon \underline{p}\rightarrow\underline{q}$ induces a natural transformation $\sigma^{\ast}\colon S^{(n)}_{\underline{q}}M\Rightarrow
S^{(n)}_{\underline{p}}M\circ ((\sigma^\ast)^{op}\otimes \sigma^\ast)$ defining a $n$-simplicial structure on $S^{(n)}_{\sbt}M$.

\begin{ex}\label{exhomdiags}
Consider the bimodule $\hom_A(-,-\otimes_AM)$ on the category $\mathcal{P}_A$ of finitely generated projective modules over a ring $A$, induced by an $A$-bimodule $M$. In this case  $S^{(n)}_{\underline{p}}M$ is the Abelian group of natural transformations of diagrams
\[S^{(n)}_{\underline{p}}M(X,Y)=\hom_A(X,Y\otimes_AM)\]
where both the tensor product of a diagram of $A$-modules with $M$ and sum of natural transformations are pointwise.
\end{ex}

Let $\coprod_CM$ be the following groupoid. It has the same objects of $C$, and only automorphisms. The automorphisms of $c$ is $M(c,c)$, and composition is defined by addition in $M(c,c)$. It is the disjoint union of groups
\[{\coprod}_CM=\coprod\limits_{c\in ObC}M(c,c)\]
and its nerve is the disjoint union of classifying spaces $\coprod\limits_{c\in ObC}BM(c,c)$.
This construction applied to $S^{(n)}_{\underline{p}}M$ defines a category
\[S^{(n)}(C;M)_{\underline{p}}:=\coprod\limits_{S^{(n)}_{\underline{p}}C}
S^{(n)}_{\underline{p}}M\]
The simplicial structure maps $\sigma^{\ast}\colon S^{(n)}_{\underline{q}}C\rightarrow S^{(n)}_{\underline{p}}C$ together with the natural transformations $\sigma^{\ast}\colon S^{(n)}_{\underline{q}}M\Rightarrow
S^{(n)}_{\underline{p}}M\circ ((\sigma^\ast)^{op}\otimes \sigma^\ast)$ define a $n$-simplicial category $S_{\sbt}^{(n)}(C;M)$, whose realization is denoted
\[\K(C;M)_n=|\mathcal{N}_{\sbt}S_{\sbt}^{(n)}(C;M)|=|\coprod_{X\in Ob S_{\sbt}^{(n)}C}
S_{\sbt}^{(n)}BM(X,X)|\]
There are structure maps
\[\KR(C;M)_n\wedge S^{1}\longrightarrow \K(C;M)_{n+1}\]
induced by the inclusions of $1$-simplicies $S^{(n+1)}_{(1,p_1,\dots,p_n)}=S^{(n)}_{(p_1,\dots,p_n)}$,
defined analogously as for $\K(C)$. Permuting the $S_{\sbt}$-constructions gives a $\Sigma_n$-action on $\K(C;M)_n$ compatible with the structure maps, defining a symmetric spectrum $\K(C;M)$. 

\begin{rem}
There is a category $C[M]$ with objects $\coprod_{c\in C}M(c,c)$ and morphism from $m_c$ in the $c$-summand to $m_{c'}$ in $c'$-summand
\[C[M](m_c,m_{c'})=\{f\in C(c,c')| f^\ast m_{c'}=f_\ast m_c\}\]
This has an exact structure by defining a sequence to be exact if it is exact in $C$, and as simplicial sets
\[Ob S_{\sbt}C[M]=\hom\left(\coprod_{S_{\sbt}C}S_{\sbt }M\right)\]
This induces an equivalence in $K$-theory $K(C[BM])\simeq K(C;M)$.
\end{rem}

\begin{rem}
Our notation is off by a suspension factor compared to the notation of \cite{DM}. For $C=\mathcal{P}_A$ and $M=\hom_A(-,-\otimes_AM)$ what we denote $\K(\mathcal{P}_A;M)$ is
$\K(\mathcal{P}_A;BM)$ in the notation of \cite[3.1]{DM}.
\end{rem}

Define a functor $\Psi\colon \coprod_CM\longrightarrow i(C\ltimes M)$ by the identity on objects, and the inclusion at the identity $(\id_c,-)\colon M(c,c)\rightarrow C(c,c)\oplus M(c,c)$ on morphisms. This extends to a map of symmetric spectra $\K(C;M)\longrightarrow \K(C\ltimes M)$.

\begin{theorem}\label{splitKRmodel} The map $\K(C;M)\longrightarrow \K(C\ltimes M)$ is a levelwise equivalence of symmetric spectra above spectrum level one. 
\end{theorem}

\begin{proof}
Since all our constructions are natural, $S_{\sbt}$ preserves equivalences, and the realization of simplicial spaces preserves levelwise equivalences, it is enough to show that the map of the statement is an equivalence in spectrum degree one, that is that the functor
\[\Psi\colon \coprod_{S_{\sbt}C}
S_{\sbt}M\longrightarrow iS_{\sbt}(C\ltimes M)\]
induces an equivalence on classifying spaces. Factor this map as
\[\Psi\colon\coprod_{S_{\sbt}C}
S_{\sbt}M\longrightarrow tS_{\sbt}(C\ltimes M)\longrightarrow iS_{\sbt}(C\ltimes M)\]
where $t$ is the set of isomorphisms of $S_{\sbt}(C\ltimes M)$ in the image of $\Psi$.

By the standard argument of \cite{Wald}, the second map is an equivalence on realizations, as both spaces are equivalent to the realization of the objects of $S_{\sbt}(C\ltimes M)$ (see also \cite[\S 2.3]{DGM}).

The first map is levelwise a fully faithful functor. We show that it is also essentially surjective, proving that it is a levelwise equivalence of categories.
We show that any diagram $X\in S_{p}(C\ltimes M)$ is isomorphic to $s_\ast p_\ast X$, by a natural isomorphism of the form $(\id,x_{\theta})\colon X_{\theta}\rightarrow X_{\theta}$ for some $x_\theta\in M(X_{\theta},X_{\theta})$. This is a higher version of \ref{reductionexactseq}, where we proved this result for $p=2$.
For every morphism $\theta\leq \rho$ in $Cat([1],[p])$ let us denote the corresponding map in the diagram $X$ by $(f_{\theta\rho},m_{\theta\rho})\colon X_{\theta}\rightarrow X_{\rho}$. The naturality condition for the elements $x_\theta$ is precisely
\[f_{\theta\rho}^{\ast}x_\rho+m_{\theta\rho}=(f_{\theta\rho})_\ast x_{\theta} \ \ \ \ \in M(X_{\theta},X_{\rho})\] 
for all $\theta\leq\rho$ in $Cat([1],[p])$. Any injective $\theta\colon [1]\rightarrow [p]$ can be written uniquely as $\theta=(j,j+k)$ for $1\leq k\leq p$ and $0\leq j\leq p-k$. We construct $x_{(j,j+k)}$ inductively on $k$.

For $k=1$ the elements $x_{(j,j+1)}$ lie on the diagonal of $X$, and define $x_{(j,j+1)}=0$. Since all maps $X_{(j,j+1)}\rightarrow X_{(i,i+1)}$ are zero these elements are compatible between them.

Now suppose that $x_{(i,i+l)}$ are defined for all $l<k$ and $0\leq i\leq p-l$, and let us define $x_{(j,j+k)}$. Notice that every map $X_{\theta}\rightarrow X_{(j,j+k)}$ with $\theta\neq (j,j+k)$ factors through $X_{(j,j+k-1)}$, where $x_{(j,j+k-1)}$ has already been defined. Similarly every map $X_{(j,j+k)}\rightarrow X_{\rho}$ factors through $X_{(j+1,j+k)}=X_{(j+1,(j+1)+(k-1))}$, with $x_{(j,j+k-1)}$ already defined. Thus for $x_{(j,j+k)}$ to define a natural map, it only needs to be compatible with $x_{(j,j+k-1)}$ and $x_{(j+1,j+k)}$. Denoting the corresponding maps by $(i,m)\colon X_{(j,j+k-1)}\rightarrow X_{(j,j+k)}$ and $(p,n)\colon X_{(j,j+k)}\rightarrow X_{(j+1,j+k)}$ the compatibility conditions for $x_{(j,j+k)}$ are
\[i^{\ast}x_{(j,j+k)}=-m+i_\ast x_{(j,j+k-1)} \ \ \ \ \mbox{and}\ \ \ \ p_{\ast}x_{(j,j+k)}=n+p^\ast x_{(j+1,j+k)}\]
The elements $-m+i_\ast x_{(j,j+k-1)}$ and $n+p^\ast x_{(j+1,j+k)}$ lie over the same point of $M((j,j+k-1),(j+1,j+k))$, and therefore a solution $x_{(j,j+k)}$ exists if the map $(i^\ast,p_\ast)$ surjects onto the pullback $P$ in the diagram
\[\xymatrix{
M(X_{(j+k,j+k-1)},X_{(j,j+k)})\ar@{->>}[r]^{p_\ast}\ar@{>->}[d]&M(X_{(j+k,j+k-1)},X_{(j+1,j+k)})\ar@{=}[r]\ar@{>->}[d]&M(X_{(j+k,j+k-1)},X_{(j+1,j+k)})\ar@{>->}[d]\\
M(X_{(j,j+k)},X_{(j,j+k)})\ar[r]^-{(i^{\ast},p_{\ast})}\ar@{->>}[d]_{i^\ast}&P\ar[r]\ar@{->>}[d]&M(X_{(j,j+k)},X_{(j+1,j+k)})\ar@{->>}[d]^{i^\ast}\\
M(X_{(j,j+k-1)},X_{(j,j+k)})\ar@{=}[r]& M(X_{(j,j+k-1)},X_{(j,j+k)})\ar[r]_{p_\ast}&M(X_{(j,j+k-1)},
X_{(j+1,j+k)})
}
\]
The vertical sequences are exact by exactness of $M$ for the exact sequence induced by $(j,j+k-1,j+k)\colon [2]\rightarrow [p]$. Moreover the top left map $p_{\ast}$ is surjective since $p$ is an admissible epimorphism. Therefore $(i^\ast,p_\ast)$ is surjective.
\end{proof}

\begin{rem}\label{Ktheorymodelasgroth}
Theorem \ref{splitKRmodel} can interpreted in terms of Grothendieck constructions. The equivalence of the statement factors as
\[\Psi\colon\coprod_{S_{\sbt}C}
S_{\sbt}M\longrightarrow (iS_{\sbt}C)\ltimes S_{\sbt}M\stackrel{\simeq}{\longrightarrow} iS_{\sbt}(C\ltimes M)\]
with the second map an equivalence by \ref{classsplitext}, as $S_{\sbt}(C\ltimes M)\rightarrow S_{\sbt}C$ is a levelwise split square zero extension of exact categories (the section is essentially surjective by the proof of \ref{splitKRmodel}). We saw in \ref{ltimesGrothendieck} that $(iS_{\sbt}C)\ltimes S_{\sbt}M$
is levelwise the Grothendieck construction of the composite \[iS_{p}C\stackrel{\Delta}{\longrightarrow}(iS_{p}C)^{op}\times iS_{p}C\stackrel{S_pM}{\longrightarrow} Ab\]
Similarly the category $\coprod_{S_{\sbt}C}
S_{\sbt}M$ is levelwise the Grothendieck construction of the restriction to the descrete category of objects
\[ObS_{p}C\stackrel{\Delta}{\longrightarrow}(ObS_{p}C)^{op}\times ObS_{p}C\stackrel{S_pM}{\longrightarrow} Ab\]
Theorem  \ref{splitKRmodel} can be rephrased as saying that the inclusion of objects $Ob S_{\sbt}C\longrightarrow iS_{\sbt}C$ induces an equivalence (after geometric realization) on Grothendieck constructions
\[(Ob S_{\sbt}C)\wr (S_{\sbt}M\circ\Delta)\stackrel{\simeq}{\longrightarrow} (i S_{\sbt}C)\wr (S_{\sbt}M\circ\Delta)\]
\end{rem}

Theorem \ref{splitKRmodel} provides the following description of the (levelwise) homotopy fiber spectrum
\[\widetilde{\K}(C\ltimes M)=\hofib\big(\K(C\ltimes M)\longrightarrow\K(C)\big)\]
Define a symmetric spectrum in level $n$ by the space
\[\widetilde{\K}(C;M)_n=\bigvee_{X\in Ob S_{\sbt}^{(n)}C}
BS_{\sbt}^{(n)}M(X,X)\]
with structure maps and symmetric actions defined as for $\K(C;M)$.

\begin{cor}\label{hofibKR} There is a natural equivalence of symmetric spectra
\[\widetilde{\K}(C;M)\simeq \widetilde{\K}(C\ltimes M)\]
\end{cor}

\begin{proof}
Projecting off the bimodule defines maps of $(n+1)$-simplicial sets
\[p_n\colon \mathcal{N}_{\sbt}S_{\sbt}^{(n)}(C;M)\longrightarrow Ob S_{\sbt}^{(n)}C\]
where the target has a trivial simplicial direction in the nerve. The collection of realizations of $n$-simplicial sets
$|Ob S_{\sbt}^{(n)}C|$ forms a symmetric spectrum in the standard way, and there is a commutative diagram
\[\xymatrix{\K(C;M)\ar[d]_{p}\ar[r]^-{\simeq}&\K(C\ltimes M)\ar[d]\\
\{|Ob S_{\sbt}^{(n)}C|\}\ar[r]_-{\simeq}&\K(C)
}\]
This gives an equivalence from the homotopy fiber of $p$ to $\widetilde{\K}(C\ltimes M)$. The inclusion at zero in $M(c,c)$ induces a splitting $s$ for $p$. As a general fact, the homotopy fiber of $p$ is equivalent (in homotopy groups) to the homotopy cofiber of $s$. Indeed if $p\colon E\longrightarrow B$ is a map of spectra split by a levelwise cofibration $s$, the square
\[\xymatrix{E\ar[d]_p\ar[r]&\hoc(s)\ar[d]\\
B\ar[r]&\ast
}\]
is cartesian since the projection to the homotopy cofiber $B\rightarrow E\rightarrow \hoc(s)$ is a fiber sequence. The map on vertical homotopy fibers $\hofib(p)\rightarrow \hoc(s)$ is therefore an equivalence.
Since both $s$ and its restriction on fixed points are levelwise cofibrations, the homotopy cofiber of $s$ is equivalent to its cofiber, which in our case is precisely $\widetilde{\K}(C;M)$.
\end{proof}

\begin{rem}
For the category $\mathcal{P}_A$ with bimodule $M=\hom_A(-,-\otimes_AM)$ we recover the equivalence
\[\widetilde{\K}(\mathcal{P}_{A\ltimes M})\underset{\ref{modulesoverltimes}}{\simeq}\widetilde{\K}(\mathcal{P}_{A}\ltimes M)\simeq\bigvee_{X\in Ob S_{\sbt}^{(n)}\mathcal{P}_A}\hom_{A}(X,X\otimes_ABM)\]
of \cite{DM}.
\end{rem}

\section{The categorical Dundas-McCarthy theorem}\label{DMsec}

We refer to \cite{BHM} and \cite{Witt} for the definition of topological cyclic homology $\TC$, and to \cite{ringfctrs} for the trace map $\tr\colon \K\rightarrow \TC$.

\begin{maintheorem}
Let $C$ be an exact category, and $M\colon C^{op}\otimes C\rightarrow Ab$ an additive functor exact in both variables. The square
\[\xymatrix{ \K(C\ltimes M)\ar[d]\ar[r]^{\tr}&\TC(C\ltimes M)\ar[d]\\
K(C)\ar[r]_{\tr}&\TC(C)
}
\]
is homotopy cartesian, for the exact structure on $C\ltimes M$ of \ref{defexactCM}.
\end{maintheorem}

Using the description of the homotopy fiber $\widetilde{K}(C\ltimes M)$ of \ref{hofibKR}, the proof of the Dundas-McCarthy theorem of \cite{McCarthy} for the functor $\mathcal{P}_{A\ltimes M}\rightarrow \mathcal{P}_A$ adapts with almost no change to our situation. We recall the proof.

\begin{proof}
For an Abelian group $G$ and a finite pointed set $X$, let $G(X)$ be the Abelian group
\[G(X)=(\bigoplus_{x\in X}Gx)/G\ast\]
This construction is functorial in $G$, and 
precomposing it with $M\colon C^{op}\otimes C\rightarrow Ab$ defines a bimodule $M(X)\colon C^{op}\otimes C\rightarrow Ab$, with corresponding exact category $C\ltimes M(X)$.
Any functor $F$ from exact categories to the category $\Sp^{\Sigma}$ of symmetric spectra induces a functor $F(C\ltimes M(-))\colon Set_{\ast}^f\rightarrow \Sp^{\Sigma}$ on the category of finite sets. This extends to a functor from finite pointed simplicial sets to simplicial symmetric spectra. For any finite pointed simplicial set $X\in sSet_{\ast}^f$ define $F(C\ltimes M(X))$ as the realization spectrum
\[F(C\ltimes M(X))=|[k]\longmapsto F(C\ltimes M(X_k))|\]
The projection functor $C\ltimes M(X)\rightarrow C$ induces a map $F(C\ltimes M(X))\rightarrow F(C)$, and define 
\[\widetilde{F}(C\ltimes M(X))=\hofib\big(F(C\ltimes M(X))\longrightarrow F(C)\big)\]
as functor $\widetilde{F}(C\ltimes M(-))\colon sSet_{\ast}^f\rightarrow \Sp^\Sigma$.

In \cite[4.2]{McCarthy} McCarthy constructs a commutative diagram
\[\xymatrix{\widetilde{\K}(C\ltimes M(X))\ar[dr]\ar@{-->}[]!<0ex,-2ex>;[ddr]!<-7ex,1ex>\ar[rr]^{\widetilde{\tr}}&&\widetilde{\TC}(C\ltimes M(X))\ar[dl]\ar@{-->}[]!<0ex,-2ex>;[ddl]!<7ex,1ex>\\
&\widetilde{\THH}(C\ltimes M(X))\ar@{-->}[d]\\
&\THH(C; M(X\wedge S^1))
}\]
where the dashed maps are weak maps. 
The left dashed map corresponds under the equivalence of \ref{hofibKR} to the inclusion of wedges into sums
\[\bigvee\limits_{c\in Ob S_{\sbt}^{(n)}C}\!\!\!\!\!BM(c,c)(X)=\!\!\!\!\!\bigvee\limits_{c\in Ob S_{\sbt}^{(n)}C}\!\!\!\!\!M(c,c)(X\wedge S^1)\rightarrow \bigoplus\limits_{c\in Ob S_{\sbt}^{(n)}C}\!\!\!\!\!M(c,c)(S^1\wedge X)\stackrel{\simeq}{\rightarrow} \THH(C;M(X\wedge S^1))\]
where the last map is the inclusion of the $0$-simplicies, which is an equivalence by \cite[1.3.3.1, \S 4.3.4]{DGM}. The composite is then roughly twice as connected as the connectivity of $X$, and therefore it induces an equivalence on Goodwillie differentials
\[D_\ast\widetilde{\K}(C\ltimes M(-))\simeq D_\ast\THH(C; M(-\wedge S^1))\]
For the right-hand dashed map, let $H$ denote the Eilemberg-MacLane functor from categories enriched in Abelian groups to categories enriched in symmetric spectra. The inclusion of wedges into direct sums defines an enriched functor $HC\vee HM\rightarrow H(C\ltimes M)$ inducing an equivalence $\THH(HC\vee HM)\stackrel{\simeq}{\rightarrow} \THH(C\ltimes M)$ over $\THH(C)$. Moreover $\THH(HC\vee HM)$ splits as
\[\THH(HC\vee HM)\simeq\bigvee_{a\geq 0}\THH_a(C;M)\]
where $\THH_a(C;M)$ is defined from the subcomplex of $\THH(HC\vee HM)$ of wedge components with $M$ appearing in $a$ smash factors, as in \cite{STC}. Using this decomposition, Hesselholt's proof of \cite{STC} holds for $C\ltimes M$, showing that the right dashed map induces an equivalence on differentials at a point
\[D_\ast\widetilde{\TC}(C\ltimes M(-))\simeq D_\ast\THH(C; M(-\wedge S^1))\]
after completion at a prime $p$. By the argument of \cite[7.1.2.7]{DGM} it is an equivalence before $p$-completion.

By commutativity of the diagram above the trace map induces an equivalence $D_\ast\widetilde{\tr}\colon D_\ast\widetilde{\K}(C\ltimes M(-))\stackrel{\simeq}{\rightarrow} D_\ast\widetilde{\TC}(C\ltimes M(-))$.
Notice moreover that for any based space $B$ there is an isomorphism of exact categories 
\[C\ltimes M(B\vee X)\cong C\ltimes (M(B)\oplus M(X))\cong (C\ltimes M(B))\ltimes M(X)\]
where $M(X)$ is a bimodule on $C\ltimes M(B)$ via the projection $C\ltimes M(B)\rightarrow C$. This shows that $D_\ast\widetilde{\K}(C\ltimes M(B\vee -))\simeq D_\ast\widetilde{\TC}(C\ltimes M(B\vee -))$, and therefore by \cite[1.3(iv)]{calcI} that the differential of the trace at any finite pointed based space $B$
\[D_B\widetilde{\tr}\colon D_B\widetilde{\K}(C\ltimes M(-))\stackrel{\simeq}{\rightarrow} D_B\widetilde{\TC}(C\ltimes M(-))\]
is an equivalence. Both functors $\widetilde{\K}(C\ltimes M(-)),\widetilde{\TC}(C\ltimes M(-))\colon sSet_{\ast}^f\rightarrow \Sp^\Sigma$ are $(-1)$-analytic in the sense of \cite{calcI},\cite{calcII}. The first one by following the proof of \cite[\S 3]{McCarthy} for $\widetilde{\K}(C\ltimes M(-))\simeq \widetilde{\K}(C;M(-))$, and the second one following \cite[\S 1]{McCarthy} using the wedge decomposition of $\THH(C\ltimes M(X))$ above. By \cite[5.3]{calcII} the trace map must already have been an equivalence before taking the differential
\[\widetilde{\tr}\colon \widetilde{\K}(C\ltimes M(X))\stackrel{\simeq}{\longrightarrow} \widetilde{\TC}(C\ltimes M(X))\]
for any ($(-1)$-connected) space $X$. For $X=S^0$ this is the equivalence $\widetilde{\tr}\colon \widetilde{\K}(C\ltimes M)\stackrel{\simeq}{\rightarrow} \widetilde{\TC}(C\ltimes M)$, and the square of the statement is cartesian.
\end{proof} 

\section{More objects}

The construction $C\ltimes M$ of \S\ref{cltimesmsec} makes sense in much greater generality. Given a functor $M\colon C^{op}\times C\rightarrow Cat$ define a category $C\ltimes M$ with objects pairs $(c\in Ob C,x\in Ob M(c,c))$ and morphism sets
\[\hom((c,x),(d,y))=\{f\colon c\rightarrow d,\gamma\colon f_\ast x\rightarrow f^\ast y \in M(c,d)\}\]
The identity of $(c,x)$ is the map $(\id_c,\id_x)$, and composition of 
$(c,x)\stackrel{(f,\gamma)}{\rightarrow}(d,y)\stackrel{(g,\delta)}{\rightarrow}(e,z)$ is defined by
\[(g,\delta)\circ(f,\gamma)=(g\circ f,g_\ast f_\ast x\stackrel{g_\ast\gamma}{\rightarrow}g_\ast f^\ast y=f^\ast g_\ast y\stackrel{f^\ast \delta}{\rightarrow}f^\ast g^\ast z)\]
When $M$ factors through Abelian groups we recover the construction of \S\ref{cltimesmsec}. As in  \ref{ltimesGrothendieck}, the isomorphism groupoid $i (C\ltimes M)$ is isomorphic to the Grothendieck construction of the composite
\[iC\stackrel{\Delta}{\longrightarrow}iC^{op}\times iC\stackrel{M}{\longrightarrow} Cat\]
This construction comes with a natural projection map $C\ltimes M\rightarrow C$.

If $C$ is an exact category one might wonder how far we can push the techniques of the previous section to prove a Dundas-McCarthy theorem for the functor $C\ltimes M\rightarrow C$. We are going to argue that one cannot do much better than the main theorem of \S\ref{DMsec}. First of all, in order to get an exact structure on $C\ltimes M$ one needs to assume the following:

\begin{itemize}
\item A zero object on $C\ltimes M$: This requires $M$ to be reduced in both variables, that is that $M(c,0)=M(0,c)=\ast$ is the trivial category for every $c\in C$. This implies that every category $M(c,d)$ comes equipped with a distinguished object $\ast=M(c,0)\rightarrow M(c,d)$ induced by the unique map $0\rightarrow d$. That is, $M$ must take values in pointed categories.
\item An $(Ab,\otimes)$-enrichement on $C\ltimes M$: this needs every category $M(c,d)$ to be enriched in the symmetric monoidal category $(Ab,\oplus)$ of Abelian groups and direct sums.
\item A family of exact sequences on $C\ltimes M$: following \S\ref{cltimesmsec} declare a sequence $(a,x)\stackrel{(i,\gamma)}{\rightarrow}(b,y)\stackrel{(q,\delta)}{\rightarrow}(c,z)$ to be exact if the underlying sequence $a\stackrel{i}{\rightarrow}b\stackrel{q}{\rightarrow}c$ is exact in $C$ and the composite $i^\ast\delta\circ q_\ast\gamma$ factors as the composite
\[(i_c\circ t_a)_\ast x\stackrel{(i_c)_\ast \id_\ast}{\longrightarrow}(i_c)_\ast (t_a)^\ast\ast \stackrel{(t_a)^\ast \id_\ast}{\longrightarrow}(i_c\circ t_a)^\ast z\]
where $i_c\colon 0\rightarrow c$ and $t_a\colon a\rightarrow 0$ are the unique maps and $(i_c)_\ast (t_a)^\ast\ast\in M(a,c)$ is the distinguished object. By trying to run the proof of \ref{CltimesMexact} for showing that this defines indeed an exact structure, one needs to invert the maps in $M(a,c)$. Thus the functor $M$ needs to take values in groupoids. Moreover $M$ needs to be exact in both variables, in the sense that for every exact sequence $a\stackrel{i}{\rightarrow}b\stackrel{q}{\rightarrow}c$ in $C$ and every object $d\in C$ and pair of objects $x,y\in M(d,a)$ the sequence
\[\hom_{M(d,a)}\stackrel{i_\ast}{\longrightarrow} \hom_{M(d,b)}(i_\ast x,i_\ast y)\stackrel{q_\ast}{\longrightarrow}\hom_{M(d,c)}(q_\ast i_\ast x,q_\ast i_\ast y)\]
is a short exact sequence of Abelian groups, and a similar condition for $M(-,d)$.
\end{itemize}

To sum up, the techniques of \ref{CltimesMexact} allow one to show that $C\ltimes M$ has a canonical exact structure induced from $C$, provided that $M\colon C^{op}\times C\rightarrow (Ab,\oplus)$-$\mathcal{G}_\ast$ takes values in the category $(Ab,\oplus)$-$\mathcal{G}_\ast$ of pointed $(Ab,\oplus)$-enriched groupoids, and it is exact in both variables.

In comparing the relative $K$-theory and the relative $\TC$ of the map $C\ltimes M\rightarrow C$ one needs to study the analytical properties of the functor
\[\bigvee_{S_{\sbt}C}\mathcal{N}M(c,c)(-)\colon sSet^{f}_\ast\longrightarrow \Sp^\Sigma\]
defined by taking configurations with labels in the morphism groups, a construction similar to \S\ref{DMsec}. This functor needs to be $(-1)$-analytic, as we are interested in its value on the $(-1)$-connected space $S^0$. The functor above is unfortunately at most $0$-analytic, unless the groupoid $M(c,c)$ is connected, so that the nerve shifts connectivity up by one. This means that we can prove the Dundas-McCarthy theorem for $C\times M\rightarrow C$ through calculus of functors only when $M$ takes values in the category $(Ab,\oplus)$-$\mathcal{G}^{c}_\ast$ of $(Ab,\oplus)$-enriched pointed connected groupoids. Evaluation at the distinguished object defines an equivalence of categories between $(Ab,\oplus)$-$\mathcal{G}^{c}_\ast$ and the category of $(Ab,\oplus)$-enriched groups. One can easily see that the two group laws on a $(Ab,\oplus)$-enriched group must be the same. Thus we recover the situation of the previous sections where $M\colon C^{op}\times C\rightarrow Ab$ takes values in Abelian groups.

\bibliographystyle{amsalpha}
\bibliography{CatDM}

\end{document}